\documentclass[12pt]{amsart}
\usepackage{amssymb}
\usepackage{color}
\usepackage{amsmath,epic,curves}
\usepackage[english]{babel}
\usepackage{graphicx}
\newtheorem{claim}{}[section]
\newtheorem{Theorem}[claim]{Theorem}

\newtheorem{Lemma}[claim]{Lemma}
\newtheorem{Proposition}[claim]{Proposition}

\renewenvironment{proof}{\noindent{\it Proof. \hskip0pt}}
                      {$\square$\par\medskip}

\begin{document}
\baselineskip 6.0 truemm
\parindent 1.5 true pc

\title{Notes on extremality of the Choi map}

\author{Kil-Chan Ha}
\address{Faculty of Mathematics and Applied Statistics, Sejong University, Seoul 143-747, Korea}

\date\today

\thanks{partially supported by NRFK 2012-0002600}



\begin{abstract}
It is widely believed that the Choi map generates an extremal ray in the cone $\mathcal P(M_3)$ of all positive linear maps between $C^*$-algebra $M_3$ of all $n\times n$ matrices over the complex field. But the only proven fact is that the Choi map generates the extremal ray in the cone of all positive linear map preserving all real symmetric $3\times 3$ matrices. In this note, we show that the Choi map is indeed extremal in the cone $\mathcal P(M_3)$. We also clarify some misclaims about the correspondence between positive semi-definite biquadratic real forms and postive linear maps,  and discuss possible  positive linear maps which coincide with the Choi map on symmetric matrices.
\end{abstract}

\maketitle

%% main text
\section{Introduction}
\label{}
Let $M_n$ be the $C^*$-algebra of all $n\times n$ matrices over the complex field. Because the convex structure of the positive cone $\mathcal P(M_n)$ of all positive linear maps between $M_n$ is highly complicated even in lower dimensions, it would be very useful to find extreme rays of this cone. Another approach to understand the convex structure of $\mathcal P(M_n)$ is to considered the possibility of decomposition of $\mathcal P(M_n)$ into subcones. For example, a positive linear map between matrix algebra is said to be decomposable if it is the sum of a completely positive linear map and a completely copositive linear map.

In the sixties, it was shown that every positive linear map in $\mathcal P(M_2)$ is decomposable, and all extreme points of the convex set of unital positive linear maps in $\mathcal P(M_2)$ had been found \cite{stormer63}. The first example of indecomposable positive linear map was given by Choi \cite{choi75,choi80}. This Choi map is defined by 
\[
\Phi(X)=
\begin{pmatrix} 
x_{11}+x_{33} & -x_{12} & -x_{13}\\ -x_{21} & x_{22}+x_{11} & -x_{23}\\ -x_{31} & -x_{32} & x_{33}+x_{22}
\end{pmatrix},
\]
where $X=(x_{ij})\in M_3$. It is widely believed \cite{robertson,osaka,kim_kye,sengupta} that the Choi map generates an extremal ray in $\mathcal P(M_3)$. It is not as trivial as one may think. We note that  extremality of some variants of Choi map  can be confirmed from their exposedness \cite{ha_kye,ha_kye2}.
But the only proven fact on the Choi map is that, for $x=(x_1,x_2,x_3)^{\rm t},\, y=(y_1,y_2,y_3)^{\rm t}\in \mathbb R^3$, the corresponding real form $y^{\rm t}\Phi(x x^{\rm t})y$  is extremal in the convex cone of all  positive semi-definite biquadratic  real forms \cite{choi77}.  Let $S_n$ be the real vector space of  all $n\times n$ real symmetric matrices in $M_n$. Then  we see that  the restriction map $\Phi|_{S_3}$ of $\Phi$ is extremal in the cone $\mathcal P(S_3)$ of all positive linear map between $S_3$.

We note that the extremality of $\Phi|_{S_3}$ in the cone $\mathcal P(S_3)$ does not imply the extremality of $\Phi$ in $\mathcal P(M_3)$. To explain this, we consider a positive linear map $\Psi_1$ defined by 
\begin{equation}\label{map_s3}
\Psi_1=\frac 12 (\Phi +\Phi \circ {\rm t}).
\end{equation}
Since $\Phi(S_3)\subset S_3$, we see that $\Psi_1 (S_3)\subset S_3$ and $\Psi_1 |_{S_3}=\Phi|_{S_3}$. Thus the restriction map $\Psi_1|_{S_3}$ is extremal in $\mathcal P(S_3)$. But $\Psi_1$ is not extremal in $\mathcal P(M_n)$. 

The purpose of this note is to clarify this situation. In the next section, we show that the Choi map is indeed extremal in $\mathcal P(M_3)$. In section 3, we explain briefly the correspondence between positive semi-definite real biquadratic forms and positive linear maps, and clarify some misclaims about this correspondence in the literatures. We also discuss  possible extensions of $\Phi|_{S_3}$ to  positive linear maps in $\mathcal P (M_3)$.

Throughout this note, $M_n(\mathbb R)$ denotes the real vector space consisting of $n\times n$ real matrices, and $\{E_{k\ell}\}$ the usual matrix units in $M_n$. For a $n\times n$ matrix $A$, $\text{det}(A)$ denotes the determinant of $A$, and $\text{d}_{k}(A)$ denotes the determinant of the submatrix formed by deleting the $k$-th row and $k$-th column of $A$.

\section{Extremality of the Choi map}
In this section, we show that the Choi map $\Phi$ is extremal in $\mathcal P(M_3)$. Suppose that 
\begin{equation}\label{eq:ext}
\Phi =\phi_1 +\phi_2
\end{equation} 
for $\phi_1,\phi_2\in \mathcal P(M_n)$.  For each $\phi_k \ (k=1,2)$, we define two linear maps $\phi_{k1}, \phi_{k2}$ by 
\[
\phi_{k1}(X)=\frac 1 2 (\phi_k (X)+\overline {\phi_k (X)}),\quad \phi_{k2}(X)=\frac 1 {2i} (\phi_k (X)-\overline {\phi_k (X)}),\quad  X\in M_3
\]
where $\overline {\phi_k (X)}$ denotes the matrix whose entries are conjugates of the corresponding entries of  the matrix $\phi_k(X)$. Then, we see that 
\[
\phi_{k\ell} (M_n(\mathbb R))\subset M_n(\mathbb R)\ (k=1,2 \text{ and } \ell =1,2)
\]
and $\phi_{k1}$ is positive linear map for $k=1,2$. Therefore, both $\phi_{11}$ and $\phi_{21}$ are positive linear maps preserving $S_3$. 

We note that 
\begin{equation}\label{map:phi}
\phi_k=\phi_{k1}+i \phi_{k2}\ (k=1,2)
\end{equation} 
and $\phi (S_3)\subset S_3$. So we see that
\[
\phi|_{S_3}=\phi_{11}|_{S_3}+\phi_{21}|_{S_3}.
\]
Thus, we can conclude that $\phi_{11}|_{S_3}=\lambda \Phi|_{S_3}$ and $\phi_{21}|_{S_3}=(1-\lambda)\Phi|_{S_3}$ for some $0\le \lambda \le 1$ because $\Phi$ is extremal in $\mathcal P(S_3)$.
From this, we have that
\begin{equation}\label{map:dec_phi}
\begin{aligned}
\phi_{11}(E_{11})&=\lambda (E_{11}+E_{22}),&\quad \phi_{21}(E_{11})&=(1-\lambda)(E_{11}+E_{22}),\\
\phi_{11}(E_{22})&=\lambda (E_{22}+E_{33}),&\quad \phi_{21}(E_{22})&=(1-\lambda)(E_{22}+E_{33}),\\
\phi_{11}(E_{33})&=\lambda (E_{33}+E_{11}),&\quad \phi_{21}(E_{33})&=(1-\lambda)(E_{33}+E_{11}),\\
\phi_{11}(S_{12})&=-\lambda S_{12},&\quad \phi_{21}(S_{12})&=-(1-\lambda) S_{12},\\
\phi_{11}(S_{23})&=-\lambda S_{23},&\quad \phi_{21}(S_{23})&=-(1-\lambda) S_{23},\\
\phi_{11}(S_{31})&=-\lambda S_{31},&\quad \phi_{21}(S_{31})&=-(1-\lambda) S_{31},
\end{aligned}
\end{equation}
where $S_{k\ell}=E_{k\ell}+E_{\ell k}$ for $1\le k<\ell \le 3$. 

For each $1\le k<\ell \le 3$, we define hermitian matrices $H_{k\ell}=(E_{k\ell}-E_{\ell k}) i$. Then $M_3$ is generated by $\mathcal B=\{ E_{11}, E_{22},E_{33},S_{12},S_{13},S_{23},H_{12},H_{13},H_{23}\}$.
We note that a positive linear map in $\mathcal P(M_3)$ is uniquely determined by its value on $\mathcal B$. Now,  we examine $\phi_1(X)$ and $\phi_2 (X)$ for each $X\in \mathcal B$ to determine two positive linear maps $\phi_1$, $\phi_2$.
First, we consider the positive semi-definite (PSD) matrices $\phi_k(E_{11}),\, \phi_k(E_{22})$ and $\phi_k (E_{33})$ for $k=1,2$. 
\begin{Lemma}\label{lem:1} $\phi_1 (E_{kk})$ and $\phi_2(E_{kk})$ are of the following form
\begin{equation}\label{eq:phikk}
\begin{aligned}
&\phi_1(E_{11})=\begin{pmatrix} \lambda & a_1 i  & 0\\- a_1 i & \lambda & 0 \\ 0 & 0 & 0\end{pmatrix},\quad 
\phi_2(E_{11})=\begin{pmatrix} 1-\lambda & -a_1 i & 0\\a_1 i & 1-\lambda & 0 \\ 0 & 0 & 0
\end{pmatrix},\\
&\phi_1(E_{22})=\begin{pmatrix} 0 & 0 & 0\\0 & \lambda & a_2 i \\0 &-a_2 i & \lambda \end{pmatrix},\quad 
\phi_2(E_{22})=\begin{pmatrix} 0 & 0 & 0\\0 & 1-\lambda & -a_2 i \\0 & a_2 i &1-\lambda \end{pmatrix},\\
&\phi_1(E_{33})=\begin{pmatrix} \lambda & 0 & a_3 i\\0 & 0 & 0 \\ -a_3 i & 0 & \lambda \end{pmatrix},\quad 
\phi_2(E_{33})=\begin{pmatrix} 1-\lambda & 0 & -a_3 i\\0 & 0 & 0 \\ a_3 i & 0 & 1-\lambda \end{pmatrix}
\end{aligned}
\end{equation}
for real numbers $a_1,\, a_2$ and $a_3$.
\end{Lemma}
\begin{proof}
Since $(3,3)$-entry of the PSD matrix $\Phi(E_{11})$ is equal to zero, $(3,3)$-entries of both $\phi_1(E_{11})$ and $\phi_2(E_{11})$ are also equal to zero from the equation \eqref{eq:ext} and the positivity of $\phi_k$. Again, the positivity of $\phi_k(E_{11})$ implies that $\phi_k(E_{11})$ is of the form in equation \eqref{eq:phikk}. The rest can be checked similarly.
\end{proof}

Note that $\phi_1 (S_{k\ell})$ and $\phi_2(S_{k\ell})$ are hermitian matrices. So, all diagonal entries of both $\phi_{12}(S_{k\ell})$ and $\phi_{22}(S_{k\ell})$ are equal to zero for all $S_{k\ell}\in \mathcal B$ by the identity \eqref{map:phi} since $\phi_{12}(S_{k\ell})$ and $\phi_{22}(S_{k\ell})$ are real matrices. 
Thus, from \eqref{map:phi} and \eqref{map:dec_phi}, we see that  $\phi_1(S_{k\ell})$ and $\phi_2(S_{k\ell})$ are of the form
\begin{equation}\label{eq:phi_sij}
\begin{aligned}
&\phi_1(S_{12})=\begin{pmatrix} 0 & -\lambda+b_1 i & b_2 i\\-\lambda-b_1 i & 0 & b_3 i\\-b_2 i & -b_3 i &0\end{pmatrix},\\
&\phi_1(S_{13})=\begin{pmatrix} 0 & b_4 i & -\lambda+b_5 i \\-b_4 i & 0 & b_6 i \\-\lambda-b_5 i & -b_6 i & 0\end{pmatrix},\\
&\phi_1(S_{23})=\begin{pmatrix} 0 & b_7 i & b_8 i\\-b_7 i & 0 & -\lambda+b_9 i\\-b_8i & -\lambda-b_9 i &0\end{pmatrix},\\
\end{aligned}
\end{equation}
\begin{equation}\label{eq:phi_sij2}
\begin{aligned}
&\phi_2(S_{12})=\begin{pmatrix} 0 & -1+\lambda-b_1 i & -b_2 i\\-1+\lambda+b_1 i & 0 & -b_3 i\\b_2 i & b_3 i &0\end{pmatrix}\\
&\phi_2(S_{13})=\begin{pmatrix} 0 & -b_4 i & -1+\lambda-b_5 i \\b_4 i & 0 & -b_6 i \\-1+\lambda+b_5 i & b_6 i & 0\end{pmatrix},\\
&\phi_2(S_{23})=\begin{pmatrix} 0 & -b_7 i & -b_8 i\\b_7 i & 0 & -1+\lambda-b_9 i\\b_8i & -1+\lambda+b_9 i &0\end{pmatrix}
\end{aligned}
\end{equation}
for real numbers $b_1,b_2,\cdots,b_9$.

Now, for any $x=(x_1,x_2,x_3)^{\rm t}\in \mathbb C^3$, we define a rank one PSD matrix
\begin{equation}\label{eq:rank1}
X[x_1,x_2,x_3]:=x x^*=(x_k \overline{x_\ell}).
\end{equation} 
From the positivity of $\phi_k(X[x_1,x_2,x_3])$, we can correlate the variables $b_k$'s in \eqref{eq:phi_sij} with the variables $a_{\ell}$'s in \eqref{eq:phikk}.
\begin{Lemma}\label{lem:2} Let $b_1,b_2,\cdots,b_9$ be the variables in \eqref{eq:phi_sij}, and $a_1,a_2,a_3$ be the variables in \eqref{eq:phikk}. Then we have 
\begin{equation}\label{eq:corr}
\begin{aligned}
&b_1=0,\ & &b_2=-a_2,\ & &b_3=-a_3,\\ 
&b_4=a_2,\ &  &b_5=0,\ & &b_6=a_1,\\ 
&b_7=-a_3,\ & &b_8=-a_1,\ & &b_9=0.
\end{aligned}
\end{equation}
\end{Lemma}
\begin{proof}
For a real number $t$, we consider two PSD matrices 
\[
\phi_k(X[1,t,0])=\phi_1(E_{11})+t \phi_1(S_{12})+t^2 \phi_1 (E_{22})\quad (k=1,2).
\]
We know that the principal minors $\text{d}_k\left (\phi_\ell(X[1,t,0])\right)$'s are nonnegative. In particular, we have 
\[
\text{d}_{3}\left( \phi_1(X[1,t,0])\right)=-b_1^2 t^2 -2a_1b_1 t+(\lambda^2-a_1^2)\ge 0
\]
for all real number $t$. Thus we get $b_1=0$. Similarily, by considering the principal minors $\text{d}_2\left( \phi_1(X[1,0,t])\right)$ and $\text{d}_1\left (\phi_1(X[0,1,t])\right)$, we can show that $b_5=0$ and $b_9=0$ respectively. 

We see that determinants $\text{det}\left(\phi_k (X[1,t,0])\right)\ (k=1,2)$ are quartic polynomial in $t$ and divisible by $t^2$. Therefore, the coefficients of $t^4$ in $\text{det}\left(\phi_k (X[1,t,0])\right)$ $(k=1,2)$
should be nonnegative. So, we have 
\[
-\lambda (a_2+b_2)^2\ge 0 \ \text{ and } \ -(1-\lambda) (a_2+b_2)^2\ge 0.
\]
Therefore, we can conclude that $b_2=-a_2$.
By the same method, we can show that $b_7=-a_3$ by considering two quartic polynomials $\text{det}\left(\phi_k (X[0,1,t])\right)$ $(k=1,2)$ in variable $t$.

Two quartic polynomials  $\text{det}\left (\phi_\ell(X[1,0,t])\right)$ in $t$ are also divisible by $t^2$. When $b_5=0$, the coefficients of $t^2$ are
\[
-\lambda (a_1-b_6)^2\ \text{ and }\ -(1-\lambda)(a_1-b_6)^2.
\]
Therefore, we get $b_6=a_1$.

Up to now, we have shown that $b_1=b_5=b_9=0$, $b_2=-a_2,\ b_6=a_1$ and $b_7=-a_3$. 
From the equations \eqref{eq:phikk}, \eqref{eq:phi_sij} and \eqref{eq:phi_sij2}, we can compute that
\[
\begin{aligned}
\text{det}&\left (\phi_1( X[1,1,1])\right)+\text{det}\left (\phi_1(X[1,1,-1])\right)+\text{det}\left (\phi_1(X[1,-1,1])\right)\\
&+\text{det}\left (\phi_1(X[1,-1,-1])\right)=-8\lambda\left( (a_1+b_8)^2+(a_2-b_4)^2+(a_3+b_3)^2\right),\\
\text{det}&\left (\phi_2( X[1,1,1])\right)+\text{det}\left (\phi_2(X[1,1,-1])\right)+\text{det}\left (\phi_2(X[1,-1,1])\right)\\
&+\text{det}\left (\phi_2(X[1,-1,-1])\right)=-8(1-\lambda)\left( (a_1+b_8)^2+(a_2-b_4)^2+(a_3+b_3)^2\right).
\end{aligned}
\]
Since the above two values must be nonnegative, we can conclude that 
\[
b_3=-a_3,\ b_4=a_2\ \text{  and }\ b_8=-a_1.
\]
This completes the proof.
\end{proof}

Now, we consider hermitian matrices $\phi_1 (H_{k\ell})$ and $\phi_2(H_{k\ell})$ for $H_{k\ell}\in \mathcal B$. For real numbers $c_i$'s and complex numbers $\alpha_i$'s, we may write $\phi_k(H_{12})$ $(k=1,2)$ as
\[
\phi_1(H_{12})=\begin{pmatrix} c_1 & \alpha_1 & \alpha_2 \\ \overline{\alpha}_1 & c_2 & \alpha_3\\
\overline{\alpha}_2 & \overline{\alpha}_3 & c_3\end{pmatrix}, \quad 
\phi_2(H_{12})=\begin{pmatrix} -c_1 & -i-\alpha_1 & -\alpha_2 \\ i-\overline{\alpha}_1 & -c_2 & -\alpha_3\\
-\overline{\alpha}_2 & -\overline{\alpha}_3 & -c_3\end{pmatrix}.
\]
We consider PSD matrices 
\[
\phi_k(X[1, -t i,0])=\phi_k(E_{11})+t \phi_k(H_{12})+t^2 \phi_k(E_{22}) \quad (k=1,2).
\]
We see that $\text{d}_3\left(\phi_k(X[1,-t i,0])\right)\ (k=1,2)$ are cubic polynomial in $t$, and the coefficients of $t^3$ are
\[
\lambda c_1 \ \text{ and }\ (\lambda -1) c_1.
\]
Therefore, we have $c_1=0$. Then we have the following: 
\[
\begin{aligned}
\text{d}_3\left(\phi_1(X[1,-t i,0])\right)= &\lambda^2-a_1^2+[\lambda c_2-2a_1 \text{Im}(\alpha_1)] t +(\lambda^2-|\alpha_1|^2) t^2 \ge 0 , \\
\text{d}_3\left(\phi_2(X[1,-t i,0])\right)= &(\lambda-1)^2-a_1^2+[(\lambda-1) c_2-2a_1 (1+\text{Im}(\alpha_1))] t  \\ &+[(\lambda-1)^2-\text{Re}(\alpha_1)^2-(1+\text{Im}(\alpha_1))^2] t^2 \ge 0,
\end{aligned}
\]
for all $t\in \mathbb R$, 
where $\alpha_1=\text{Re}(\alpha_1)+\text{Im}(\alpha_1) i$.
So we get the conditions 
\begin{gather}
\lambda^2\ge |\alpha_1|^2 \label{cond1}\\
(\lambda-1)^2\ge \text{Re}(\alpha_1)^2+(1+\text{Im}(\alpha_1))^2 \label{cond2}
\end{gather}
Since $0\le \lambda \le 1$, we see that $\text{Im}(\alpha_1)<0$ by \eqref{cond2}. Then, we get that $-\lambda\le \text{Im}(\alpha_1)<0$ from the condition \eqref{cond1}, that is, $0\le 1-\lambda \le 1+\text{Im}(\alpha_1)<1$. Therefore, we have that
\[
(1-\lambda)^2\le (1+\text{Im}(\alpha_1))^2\le \text{Re}(\alpha_1)^2+(1+\text{Im}(\alpha_1))^2\le (1-\lambda)^2.
\]
Consequently, we conclue $\alpha_1=-\lambda i$. We note that this implies the coefficient of $t$ in $\text{d}_3\left(\phi_1(X[1,-t i,0])\right)$ should be zero. Therefore, we get $c_2=-2a_1$.

When $c_1=0$, 
 $\text{d}_2\left(\phi_k(X[1,-t i,0])\right)\ (k=1,2)$ are quadratic polynomials in $t$ divisible by $t$. Therefore, the coefficients of $t$ should be zero. From this observation, we have that
\[
\lambda c_3=0 \ \text{ and }\ (\lambda-1)c_3=0.
\]
Consequently, we see that $c_3=0$.

Finally, we show that $\alpha_2=a_2$ by considering the determinant of $\phi_k(X[1,-t i,0]) $. Under the conditions $c_1=0$ and $\alpha_1=-\lambda i$, we can show that the determinants $\text{det}\left (\phi_k(X[1,-t i,0]) \right)$ $(k=1,2)$ are quartic polynomials in $t$, and the coefficients of $t^4$ are 
\[
-\lambda |\alpha_2-a_2|^2 \ \text{ and } \ -(1-\lambda)|\alpha_2-a_2|^2.
\]
Since both coefficients should be nonnegative, we get $\alpha_2=a_2$.

To sum up, we have correlated all entries of $\phi_k(H_{12})$ except $\alpha_3$ with $a_i$ in \eqref{eq:phikk}. That is, $c_1=c_3=0,c_2=-2a_1,\alpha_1=-\lambda i$ and $\alpha_2=a_2$.
\begin{Lemma}\label{lem:3} Let $a_1,a_2,a_3$ be the real variables in \eqref{eq:phi_sij}. Then $\phi_1(H_{k\ell})$ and $\phi_2(H_{k\ell})$ are of the following forms
\begin{equation}\label{eq:phi_hij}
\begin{aligned}
&\hskip-0.4truecm \phi_1(H_{12})=\begin{pmatrix} 0 & -\lambda i & a_2 \\ \lambda i & -2a_1 & \alpha\\ a_2 & \overline{\alpha} & 0\end{pmatrix}, & \hskip-0.4truecm
\phi_2(H_{12})=\begin{pmatrix} 0 & (\lambda-1) i & -a_2 \\ (1-\lambda) i &  2a_1 & -\alpha\\ -a_2 & -\overline{\alpha} & 0\end{pmatrix},\\
&\hskip-0.4truecm \phi_1(H_{13})=\begin{pmatrix} -2a_3 &\beta & -\lambda i \\ \overline{\beta} & 0 & -a_1\\ \lambda i & -a_1 & 0\end{pmatrix}, & \hskip-0.4truecm
\phi_2(H_{13})=\begin{pmatrix} 2a_3 & -\beta & (\lambda-1)i \\ -\overline{\beta}&  0 & a_1\\ (1-\lambda)i & a_1& 0\end{pmatrix},\\
&\hskip-0.4truecm \phi_1(H_{23})=\begin{pmatrix} 0 & -a_3& \gamma \\ -a_3 & 0 & -\lambda i \\ \overline{\gamma} & \lambda i & -2a_2\end{pmatrix}, & \hskip-0.4truecm
\phi_2(H_{23})=\begin{pmatrix} 0 & a_3& -\gamma \\a_3 &  0 & (\lambda-1) i \\ -\overline{\gamma} & (1-\lambda) i & 2 a_2\end{pmatrix}.
\end{aligned}
\end{equation}
\end{Lemma}
\begin{proof} For the case of $\phi_k(H_{12})$, we have done it with $\alpha_3=\alpha$. Through the same process with principal minors $\text{d}_3\left( \phi_k(X[0,1,-t i])\right), \text{d}_1\left( \phi_k(X[0,1,-t i])\right)$ and the determinant $\text{det}\left( \phi_k(X[0,1,-t i])\right)$, we can show that $\phi_k(H_{23})$ is of the form in \eqref{eq:phi_hij} for $k=1,2$. 

For the case of $\phi_k(H_{13})$, it suffices to consider  $\text{d}_3\left( \phi_k(X)\right), \text{d}_2\left( \phi_k(X)\right)$ and then $\text{det}\left(\phi_k(X)\right)$ with $X=X[1,0,-t i]$. But, in this case, each $\text{d}_2\left(\phi_k(X)\right)$ is quartic polynomial in $t$ divisible by $t^2$, and we can determine $(1,3)$ and $(1,1)$ entries of $\phi_k(H_{13})$ by examining the coefficients of $t^2$ as in \eqref{cond1} and \eqref{cond2}.
\end{proof}
Now, we are ready to prove that the Choi map is indeed extremal in the cone $\mathcal P(M_3)$. We note that a positive linear map $\psi$ satisfies
\[
\psi(E_{k\ell})=\frac 1 2 \left ( \psi (S_{k\ell})-i \psi (H_{k\ell})\right),\quad
\psi(E_{\ell k})=\frac 1 2 \left ( \psi (S_{k\ell})+i \psi (H_{k\ell})\right)
\]
for $1\le k< \ell \le 3$.
Therefore, positive linear maps $\phi_1$ and $\phi_2$ in \eqref{eq:ext} are uniquely determined by the Lemma~\ref{lem:1}, \ref{lem:2} and \ref{lem:3}. For the convenience of readers, we make up a list of entries of $\phi_1(X)$ and $\phi_2(X)$ for $3\times 3$ matrix $X=(x_{k\ell})$. In the following list, $[A]_{k\ell}$ denotes the entry in the $k$-th row and $\ell$-th column of a matrix $A$.
\begin{equation}\label{phi}
\begin{aligned}
&[\phi_1(X)]_{11}  = (x_{11}+x_{33})\lambda+a_3 (x_{13}-x_{31}) i,\\
&[\phi_1(X)]_{12} = -x_{12}\lambda+a_1 x_{11} i  -a_3 x_{32} i +\frac 12 (a_2-\beta )x_{13}i+\frac 12 (a_2+\beta)x_{31}i, \\
&[\phi_1(X)]_{13} =-x_{13}\lambda+a_3 x_{33}i-a_2 x_{12} i-\frac 12(a_1+\gamma)x_{23}i-\frac 12 
(a_1-\gamma)x_{32}i,\\
&[\phi_1(X)]_{21} = -x_{21}\lambda-a_1 x_{11} i  +a_3 x_{23} i -\frac 12 (a_2-\overline{\beta} )x_{31}i-\frac 12 (a_2+\overline{\beta})x_{13}i, \\
&[\phi_1(X)]_{22} =(x_{22}+x_{11})\lambda+a_1(x_{12}-x_{21})i,\\
&[\phi_1(X)]_{23} =-x_{23}\lambda+a_1 x_{13} i+a_2 x_{22} i-\frac 12 (a_3+\alpha) x_{12} i-\frac 12 (a_3-\alpha)x_{21}i,\\
&[\phi_1(X)]_{31} =-x_{31}\lambda-a_3 x_{33}i+a_2 x_{21} i+\frac 12(a_1+\overline{\gamma})x_{32}i+\frac 12 
(a_1-\overline{\gamma})x_{23}i,\\
&[\phi_1(X)]_{32} =-x_{32}\lambda-a_1 x_{31} i-a_2 x_{22} i+\frac 12 (a_3+\overline{\alpha}) x_{21} i+\frac 12 (a_3-\overline{\alpha})x_{12}i,\\
&[\phi_1(X)]_{33}=(x_{33}+x_{22})\lambda+a_2(x_{23}-x_{32})i,\\
%\end{aligned}
%\end{equation}
%\begin{equation}\label{phi2}
%\begin{aligned}
&[\phi_2(X)]_{11}  = (x_{11}+x_{33})(1-\lambda)-a_3 (x_{13}-x_{31}) i,\\
&[\phi_2(X)]_{12} = -x_{12}(1-\lambda)-a_1 x_{11} i  +a_3 x_{32} i -\frac 12 (a_2-\beta )x_{13}i-\frac 12 (a_2+\beta)x_{31}i, \\
&[\phi_2(X)]_{13} =-x_{13}(1-\lambda)-a_3 x_{33}i+a_2 x_{12} i+\frac 12(a_1+\gamma)x_{23}i+\frac 12 
(a_1-\gamma)x_{32}i,\\
&[\phi_2(X)]_{21} = -x_{21}(1-\lambda)+a_1 x_{11} i  -a_3 x_{23} i +\frac 12 (a_2-\overline{\beta} )x_{31}i+\frac 12 (a_2+\overline{\beta})x_{13}i, \\
&[\phi_2(X)]_{22} =(x_{22}+x_{11})(1-\lambda)-a_1(x_{12}-x_{21})i,\\
&[\phi_2(X)]_{23} =-x_{23}(1-\lambda)-a_1 x_{13} i-a_2 x_{22} i+\frac 12 (a_3+\alpha) x_{12} i+\frac 12 (a_3-\alpha)x_{21}i,\\
&[\phi_2(X)]_{31} =-x_{31}(1-\lambda)+a_3 x_{33}i-a_2 x_{21} i-\frac 12(a_1+\overline{\gamma})x_{32}i-\frac 12 
(a_1-\overline{\gamma})x_{23}i,\\
&[\phi_2(X)]_{32} =-x_{32}(1-\lambda)+a_1 x_{31} i+a_2 x_{22} i-\frac 12 (a_3+\overline{\alpha}) x_{21} i-\frac 12 (a_3-\overline{\alpha})x_{12}i,\\
&[\phi_1(X)]_{33}=(x_{33}+x_{22})(1-\lambda)-a_2(x_{23}-x_{32})i
\end{aligned}
\end{equation}
where $a_1,a_2,a_3\in \mathbb R$, $\alpha,\beta,\gamma\in \mathbb C$ and $0\le \lambda \le 1$.

First, we determine real variables $a_1,a_2$ and $a_3$. To do this, we consider principal minors $\text{d}_3\left (\phi_k(X[1,t,s i])\right) \ (k=1,2)$ for real numbers $t$ and $s$, . They are quadratic polynomials in $t$ with the following coefficients of $t^2$.
\[
4\lambda a_3 s+(\lambda^2-a_3^2)s^2,\quad 4(\lambda-1)a_3 s+\left((\lambda-1)^2-a_3^2 \right)s^2
\]
These coefficients should be nonnegative for all $s$ because principal minors of a PSD matrix are nonnegative. Therefore, we have $a_3=0$. 
And then, by considering $\text{d}_2\left(\phi_k(X[t,1, s i])\right)$ as a polynomial in $t$ for each $k=1,2$, we can show that $a_2=0$ similarly. Finally, we get $a_1=0$ from the condition that the coefficients of $t^2$ in quadratic polynomials $\text{d}_1(\phi_k( X[s i, 1, t])) \ (k=1,2)$ should be nonnegative. Consequently, we have 
\begin{equation}\label{eq:a}
a_1=a_2=a_3=0.
\end{equation}

Now, we show that all complex variables $\alpha,\beta$ and $\gamma$ should be zero when $a_1=a_2=a_3=0$ in \eqref{phi}. 

\noindent {\bf (Case 1: $\lambda=0$)}

If $\lambda=0$, then we have PSD matrices
\[
\phi_1(X[1,1,i])=\begin{pmatrix} 0 & -\beta & -\gamma\\-\overline{\beta} & 0 & 0\\ -\overline{\gamma} & 0 & 0\end{pmatrix},\quad 
\phi_1(X[1,i,1])=\begin{pmatrix} 0 & 0 & \gamma\\0 & 0 & -\alpha\\ \overline{\gamma} & -\overline{\alpha} & 0\end{pmatrix}
\]
from \eqref{phi}. Therefore, we have $\alpha=\beta=\gamma=0$ when $\lambda=0$. In this case, we conclude that $\phi_2$ is the Choi map $\Phi$ and $\phi_1$ is the zero map.

\noindent {\bf (Case 2: $\lambda=1$)}

In this case, by considering PSD matrices $\phi_2(X[1,1,i])$ and $\phi_2(X[1,i,1])$, we can show that $\phi_1$ is the Choi map $\Phi$ and $\phi_2$ is the zero map as in (Case 1).

\noindent {\bf (Case 3: $0<\lambda <1$)}

Since the determinant of a PSD is nonnegative, we have the following inequalities:
\begin{equation}\label{alphabeta}
\begin{aligned}
&\text{det}\left(\phi_1(X[1,e^{\frac {\pi}2 i},e^{\frac {\pi}2 i}])\right)+
\text{det}\left(\phi_1(X[1,e^{-\frac {\pi}2 i},e^{-\frac {\pi}2 i}])\right)\\
=&-4\lambda(|\alpha|^2+|\beta|^2)+12\lambda^2 \text{Im}(\beta)\ge 0,\\
&\text{det}\left(\phi_2(X[1,e^{\frac {\pi}2 i},e^{\frac {\pi}2 i}])\right)+
\text{det}\left(\phi_2(X[1,e^{-\frac {\pi}2 i},e^{-\frac {\pi}2 i}])\right)\\
=&-4(1-\lambda)(|\alpha|^2+|\beta|^2)-12(1-\lambda)^2\text{Im}(\beta)\ge 0,\\
\end{aligned}
\end{equation}
These are equivalent to the inequalities
\[
0\le \dfrac{|\alpha|^2+|\beta|^2}{3\lambda}\le \text{Im}(\beta)\le -\dfrac{|\alpha|^2+|\beta|^2}{3(1-\lambda)}\le 0
\]
because of $0<\lambda<1$. Therefore, we see $\text{Im}(\beta)=0$, and so $\alpha=\beta=0$.

Finally, we show that $\gamma=0$ when $a_1=a_2=a_3=0$ and $\alpha=\beta=0$ in \eqref{phi}.
For a PSD matrix $X[1,1,i]$, we get two inequalities
\[
\begin{aligned}
& \text{det}\left(\phi_1(X[1,1,i])\right)=-2\lambda |\gamma|^2+6\lambda^2 \text{Im}(\gamma)\ge 0,\\
& \text{det}\left(\phi_2(X[1,1,i])\right)=-2(1-\lambda)|\gamma|^2-6(1-\lambda)^2\text{Im}(\gamma)\ge 0.
\end{aligned}
\]
Then we see that
\[
0\le \dfrac{|\gamma|^2}{3\lambda}\le \text{Im}(\gamma)\le -\dfrac{|\gamma|^2}{3(1-\lambda)}\le 0.
\]
Thus, we get $\gamma=0$. 
Consequently, we have shown that $\phi_1=\lambda \Phi$ and $\phi_2=(1-\lambda)\Phi$.

In any cases, we have the following.
\begin{Theorem}
The Choi map generates an extreme ray in the cone $\mathcal P (M_3)$.
\end{Theorem}

\section{Correspondence between positive semi-definite biquadratic real forms and positive linear maps}
We note \cite{choi80} that the Choi map $\Phi$ is originated from a real biquadratic  form $B(x,y)$ with the relation
\[
\begin{aligned}
B(x,y)=&(x_1^2+x_3^2) y_1^2+(x_2^2+x_1^2)y_2^2+(x_3^2+x_2^2)y_3^2\\
 &\hskip 0.7truecm -2(x_1x_2y_1y_2+x_2x_3y_2y_3+x_3x_1y_3y_1)\\
=&y^{\rm t}\Phi(x x^{\rm t})y
\end{aligned}
\]
for real column vector $x=(x_1,x_2,x_3)^{\rm t},y=(y_1,y_2,y_3)^{\rm t}\in \mathbb R^3$.

In general, for any positive real linear map $\phi\,:\,S_n\to S_n$, we get the corresponding positive semi-definite real biquadratic form $B_{\phi}(x,y)$ with $x,y\in \mathbb R^n$ defined by 
$B_{\phi}(x,y)=y^{\rm t}\phi(x x^{\rm t})y$. 

On the other hand, let $B(x,y)$ be a positive sem-definite real biquadratic form with $x,y\in \mathbb R^n$. Then, for any fixed $x\in \mathbb R^n$, $B(x,y)$ is a positive semi-definite real quadratic form with respect to $y\in \mathbb R^n$. So we can write $B(x,y)$ in the form $y^{\rm t} S(x) y$ for some $n\times n$ symmetric matrix $S(x)$. Thus we get a map which take any one dimensional projection $x x^{\rm t}$ to $S(x)$. Consequently, we get the corresponding positive linear map $\phi:S_n\to S_n$ by linearity, which satisfy the relation $B_{\phi}(x,y)=B(x,y)$. Therefore, we see that there is a one-to-one correspondence between the set of positive semi-definite real biquadratic forms and the set $\mathcal P(S_n)$ consisting of positive real linear maps between $S_n$.  

We can find some misclaims on the above correspondence in the literatures \cite{osaka,sengupta}. They claim that 
\begin{itemize}
\item[(i)] Let $\Psi$ be a positive complex linear map with $\Psi (M_n(\mathbb R))\subset M_n(\mathbb R)$. If the corresponding real biquadratic form $B_{\Psi}(x,y)$ is extremal in the cone of all positive semi-definte real biquadratic forms, then $\Psi$ is extremal in the cone $\mathcal P(M_n)$.
\item[(ii)] Using  linearity and hermicity, the above correspondence can be extended to the correspondence between the set of positive semi-definite real biquadratic forms and the set $\mathcal P(M_n)$ trivially.
\end{itemize} 
For the first claim, we have shown there exists a counter example $\Psi_1$ in \eqref{map_s3}. Here, we clarify the claim (ii).

As before, we define $n\times n$ symmetric matrices $S_{k\ell}$ and antisymmetric matrices $A_{k\ell}$ for $1\le k<\ell \le n$ by
\[
S_{k\ell}=E_{k\ell}+E_{\ell k},\quad A_{k\ell}=E_{k\ell}-E_{\ell k}.
\]
We note that $S_n$ is generated by 
\[
\mathcal G=\{E_{kk}:1\le k\le n\}\cup \{S_{k\ell}:1\le k<\ell \le n\}
\]
and real matrix algebra $M_n(\mathbb R)$ is generated by $\mathcal G\cup \{A_{k\ell}:1\le k<\ell\le n\}$. 
Thus, a map $\phi\in \mathcal P(S_n)$ can be extended to a positive linear map $\widetilde \phi$ in $\mathcal P(M_n(\mathbb R))$ by defining $\widetilde\phi(A_{k\ell})$ for $1\le k<\ell \le n$.
It is worthy to note that the positivity of $\widetilde \phi$ is not affected by $\widetilde \phi(A_{k\ell})$'s. That is, the positivity of $\widetilde \phi$ is determined by $\widetilde\phi|_{S_n}=\phi$.

On the other hand, we can uniquely extend $\widetilde\phi\in \mathcal P(M_n(\mathbb R))$ to the complex linear map 
$\widetilde \phi$ between $M_n$ by the  linearity $\widetilde \phi (X+iY)=\widetilde \phi (X)+i\widetilde \phi(Y)$ for $X,Y\in M_n(\mathbb R)$. But, in this extension, the positivity of the complex linear map $\widetilde \phi$ is not determined by the positivity of the real linear map $\widetilde \phi$.

In general, let $\Psi$ be a positive linear map in $\mathcal P(M_n)$ with $\Psi(M_n(\mathbb R))\subset M_n(\mathbb R)$. So, we can regard $\Psi$ as an extension of real linear map. Then, it is well known that $\Psi$ preserve hermitian matrices. That is, $\Psi (i A_{k\ell})$ must be a hermitian matrix.

Since any hermitian matrix $H\in M_n$ can be written by $H=S+i A$ with a symmetric matrix $S\in  M_n(\mathbb R)$ and an antisymmetric matrix $A\in M_n(\mathbb R)$, we observe the following.
\begin{Proposition}\label{prop:hermitian} Let $\Psi$ be a positive linear map in $\mathcal P(M_n)$ with $\Psi(M_n(\mathbb R))\subset M_n(\mathbb R)$. Then $\Psi$ preserve hermitian matrices if and only if $\Psi$ preserve symmetric matrices and antisymmetric matrices respectively.
\end{Proposition} 
\begin{proof}
For a symmetric matrix $S$ and an antisymmetric matrix $A$,  we have 
\[
\Psi(S+i A)^{*}=\Psi(S)^{\rm t}-i\Psi(A)^{\rm t}=\Psi(S)+i\Psi(A)=\Psi(S+i A)
\]
since $\Psi$ preserve hermitian matrices and $\Psi(M_n(\mathbb R))\subset M_n(\mathbb R)$. Therefore, 
$\Psi(S)^{\rm t}=\Psi(S)$ and $\Psi(A)^{\rm t}=-\Psi(A)$. This  completes the proof.
\end{proof}

We note that there exists non-positive linear map $\Psi_2$ between $M_3$, which satisfies the following conditions
\begin{itemize}
\item $\Psi_2$ preserves hermitian matrices.
\item $\Psi_2|_{M_3(\mathbb R)}$ is a positive linear map between $M_3(\mathbb R)$ and $\Psi_2|_{S_3}=\Phi|_{S_3}$ for the Choi map $\Phi$.
\end{itemize}
From the condition $\Psi_1|_{S_3}=\Phi|_{S_3}$, $\Psi_2$ is determined by the values $\Psi_2(A_{k\ell})$.
We put
\[
\Psi_2(A_{12})=-A_{12},\ \Psi_2(A_{13})=-A_{13},\ \Psi_2(A_{23})=-A_{12}.
\]
Then, $\Psi_2$ is defined by 
\begin{equation}\label{map:psi2}
\Psi_2(X)=\begin{pmatrix} 
x_{11}+x_{33} & -x_{12}-\frac 12(x_{23}-x_{32}) & -x_{13}\\
-x_{21}+\frac 12(x_{23}-x_{32}) & x_{11}+x_{22} &-\frac 12 (x_{23}+x_{32})\\
-x_{31} & -\frac 12 (x_{23}+x_{32}) & x_{22}+x_{33}
\end{pmatrix}
\end{equation}
for $X=(x_{k\ell})\in M_n$.
We know that $\Psi_2|_{M_3(\mathbb R)}$ is a positive linear map between $M_3(\mathbb R)$ and $\Psi_2|_{S_3}=\Phi_{S_3}$. But, this map is not positive map between $M_3$ because  $\text{det}\left(\Psi_2(X)\right)=-25$ for a PSD matrix $X=X[1,2-i,-1-i]$ in \eqref{eq:rank1}.

Now, we give an example of extremal positive linear map $\Psi_3$ in $\mathcal P(M_3)$, which is not equal to the Choi map $\Phi$ but $\Psi_3|_{S_3}=\Phi|_{S_3}$. This example explains the claim (ii) is nonsense, and we can conclude that extremal extension of $\Phi|_{S_3}$ is not unique.  We define $\Psi_3$ by 
\[
\begin{aligned}
&\Psi_3(X)=\Phi(X) \ \text{ for all } X\in \{E_{11},E_{22},E_{33},S_{12},S_{13},S_{23},A_{12},A_{13}\},\\
&\Psi_3(A_{23})=A_{23}.
\end{aligned}
\]
Then, we obtain a complex linear map $\Psi_3$ preserving hermitian matrices by Proposition~\ref{prop:hermitian}.  We will show that this map is indeed positive. We recall that the positivity of the Choi map is easily proven through the positivity of the real map $\Phi|_{S_3}$ with the following relation:
\[
\Phi(X[x_1,x_2,x_3])
=V \Phi(X[r_1,r_2,r_3]) V^* \ \text{ with } V=\begin{pmatrix}
e^{i \theta_1} & 0 & 0\\0 & e^{i \theta_2} & 0\\0 & 0 & e^{i\theta_3}
\end{pmatrix}
\]
where $x_i=r_i e^{i\theta_i}$ for $i=1,2,3$. But, this method is not applicable to general cases. In fact, it is easy to show that there exists no matrix $V$ satisfying $\Psi_3(X)=V \Psi_3(Y) V^*$ for  rank one PSD matrices $X=X[r_1 e^{i\theta_1},r_2 e^{i\theta_2},r_3 e^{i\theta_3}]\in M_3$ and $Y=X[r_1,r_2,r_3]\in M_3(\mathbb R)$. 

So, we will show that $\Psi_3(X)$ is a PSD matrix for any rank one PSD matrix $X$. Since $\Phi$ is positive, we see that 
\[
\text{d}_k\left(\Psi_3(X)\right)=\text{d}_k\left(\Phi(X)\right) \ge 0
\]
for  any rank one PSD matrix $X=X[x_1,x_2,x_3]$ and $k=1,2,3$.
 By the arithmetic mean-geometric mean inequality, we also see that
\[
\begin{aligned}
\text{det}\left(\Psi_3(X)\right)
=&|x_1|^2|x_3|^4+|x_2|^2|x_1^4|+|x_3|^2|x_2|^4\\
&\hskip 0.8truecm -|x_1|^2|x_2|^2|x_3|^2-2|x_1|^2 \text{Re}(\overline{x}_2^2 x_3^2)\\
\ge & 2|x_1|^2\left(|x_2|^2|x_3|^2 -\text{Re}(\overline{x}_2^2 x_3^2)\right)\ge 0.
\end{aligned}
\]
Therefore, $\Psi_3(X)$ is a PSD matrix for any rank one PSD matrix $X=X[x_1,x_2,x_3]$. That is, $\Psi_3$ is a positive linear map in $\mathcal P(M_3)$. The extremality of $\Psi_3$ in $\mathcal P(M_3)$ is similarly checked as that of the Choi map.

Finally, we discuss extremality for various convex cones of positive linear maps. Let $\Psi$ be a positive linear map preserving $M_n(\mathbb R)$ in $\mathcal P(M_n)$. Then $\Psi|_{M_n(\mathbb R)}$ belongs to $\mathcal P(M_n(\mathbb R))$. Furthermore, we see that $\Psi(S_n)\subset S_n$ from the positivity of $\Psi$ and $\Psi(M_n(\mathbb R))\subset M_n(\mathbb R)$, that is, $\Psi|_{S_n}\in \mathcal P(S_n)$. Now, let $\mathcal P_{\mathbb R}(M_n)$ be the set of all positive complex linear maps preserving $M_n(\mathbb R)$. Then $\mathcal P_{\mathbb R}(M_n)$ is a convex subcone contained in $\mathcal P(M_n)$. Since any real linear map between $M_n(\mathbb R)$ can be uniquely extended to the complex linear map between $M_n$, we may think $\mathcal P_{\mathbb R}(M_n)=\mathcal P(M_n)\cap \mathcal P(M_n(\mathbb R))$.  So, we may consider extremality of a map in each convex cones $\mathcal P(M_n)$, $\mathcal P_{\mathbb R}(M_n)$, $\mathcal P(M_n(\mathbb R))$ and $\mathcal P(S_n)$. We note that the Choi map is extremal in each of these convex cones.

Owing to the map $\Psi_1$ in \eqref{map_s3}, we know that the extremality of $\Psi|_{S_n}$ in $\mathcal P(S_n)$ implies neither the extremality of $\Psi|_{M_n(\mathbb R)}$ in $\mathcal P(M_n(\mathbb R))$ nor the extremality of $\Psi$ in $\mathcal P(M_n)$. 
On the other hand, for $\Psi\in \mathcal P_{\mathbb R}(M_n)$, it is easy to see that the extremality of $\Psi|_{M_n(\mathbb R)}$ in $\mathcal P(M_n(\mathbb R))$ implies the extremality of  $\Psi$ in $\mathcal P_{\mathbb R}(M_n)$. But, it seems that the converse does not hold since there exists a non-positive complex linear map $\Psi$ satisfying two properties  $\Psi(M_n(\mathbb R))\subset M_n(\mathbb R)$ and $\Psi|_{M_n(\mathbb R)}\in \mathcal P(M_n(\mathbb R))$ as a map $\Psi_2$ in \eqref{map:psi2}. So, it is interesting to find an extremal positive linear map $\Psi$ in $\mathcal P_{\mathbb R}(M_n)$ whose restriction $\Psi|_{M_n(\mathbb R)}$ is not extremal in $\mathcal P(M_n(\mathbb R))$.

\end{document}